\documentclass{amsart}

\usepackage[margin=1.5in]{geometry}
\usepackage{amsmath}
\usepackage{amssymb}
\usepackage{amsthm}
\usepackage{enumerate}
\usepackage[colorlinks=true,
citecolor=blue,
linkcolor=red,
urlcolor=black,
anchorcolor=black]{hyperref}
\usepackage{xcolor}
\usepackage[english]{babel}
\usepackage{graphicx}
\usepackage{longtable}
\usepackage{tikz-cd}

\numberwithin{equation}{section}
\newtheorem{theorem}{Theorem}[section]
\newtheorem{prop}[theorem]{Proposition}
\newtheorem{lemma}[theorem]{Lemma}

\newtheorem{defn}[theorem]{Definition}

\theoremstyle{definition}

\newtheorem{question}[theorem]{Question}

\newcommand{\Z}{\mathbb{Z}}

\newcommand{\Q}{\mathbb{Q}}
\newcommand{\C}{\mathbb{C}}
\newcommand{\R}{\mathbb{R}}

\newcommand{\Qsquare}{\Q^\times / (\Q^\times)^2}

\newcommand{\wt}[1]{\widetilde{#1}}

\title{Cusp cross-section phenomena for arithmetic hyperbolic manifolds}
\author{Duncan McCoy}
\author{Connor Sell}
\date{}

\begin{document}
\begin{abstract}
     Although every flat manifold occurs as a cusp cross-section in at least one commensurability class of arithmetic hyperbolic manifolds, it turns out that some flat manifolds have the property that they occur as cusp cross-sections in a unique commensurability class of arithmetic hyperbolic manifolds -- a phenomenon which we will refer to as the UCC property. We construct flat manifolds with the UCC property in all dimensions $ n \geq 32 $. We also show that the number of distinct commensurability classes containing cusp cross-sections with the UCC property is unbounded. We also exhibit pairs of manifolds in all dimensions $ n \geq 24 $ that cannot arise as cusp cross-sections in the same commensurability class of arithmetic hyperbolic manifolds.

     The main tool is previous work of the authors algebraically characterizing when a given flat manifold arises as the cusp cross-section of a manifold in a given commensurability class of arithmetic hyperbolic manifolds. 
\end{abstract}
\maketitle

%\tableofcontents

%\DM{I've added some more ideas to the introduction. I'm not expecting to keep all of it, just give us more material from which we can pick the best bits. Feel free to make changes.}
\section{Introduction}
The aim of this paper is to explore two phenomena relating to cusp cross-sections of arithmetic hyperbolic manifolds. The study of cusp cross-sections dates back to the notion of geometric bounding.  In 1983, Farrell and Zdravkovska conjectured that every flat manifold bounds geometrically - that is, it occurs as the cusp cross-section of a one-cusped hyperbolic manifold \cite{FarrellZ}.  The conjecture was proven false in 2000 by Long and Reid \cite{Long-Reid2000geometric}; however, they later showed (along with McReynolds) that every flat manifold appears as a cusp cross-section of some arithmetic hyperbolic manifold, not necessarily one-cusped \cite{LongReid, McReynolds2009covers}. It is then natural to wonder under what conditions a flat manifold can be obtained as a cusp cross-section of a hyperbolic manifold.

The construction of Long-Reid and McReynolds produces a commensurability class of \emph{arithmetic} hyperbolic manifolds. The second author proved in \cite{Sell} that some flat 3-manifolds are obstructed from arising in some commensurability classes of hyperbolic 4-manifolds, raising the question of whether such obstructions occur in other dimensions, and if so, how restrictive they can be - for instance, are there flat manifolds that occur in only finitely many classes? It turns out that this question cannot possibly be addressed in full generality by the methods of \cite{LongReid, McReynolds2009covers}. In \cite{McCoySell2024}, the authors produced an algebraic condition which determines precisely whether any given flat manifold arises as a cusp cross-section in a given commensurability class of arithmetic hyperbolic manifolds. In sufficiently high dimension, all known examples of hyperbolic manifolds are either arithmetic or made from gluing arithmetic manifolds together as in \cite{Gromov}, so the arithmetic case is natural to study.

It was shown in \cite{McCoySell2024} that there are examples of flat manifolds that arise as a cusp cross-section in precisely one commensurability class of arithmetic hyperbolic manifolds. The first aim of this article is to explore the scope of this phenomenon. We say that a flat $n$-manifold has the \emph{unique arithmetic commensurability class (UCC) property} if it occurs as a cusp cross-section in a unique commensurability class of non-compact arithmetic hyperbolic $(n+1)$-manifolds. Firstly we show that manifolds with the UCC property exist in all sufficiently large dimensions.

\newcommand{\thmUCCbound}{
For all $n\geq 32$ there exists an orientable flat $n$-manifold with the UCC property. %Moreover, there is an example appearing as a cusp cross-section only in the commensurability class of  arithmetic hyperbolic manifolds defined by the form 
    % \[q(x_1, \dots, x_{n+2})=x_1^2 +\dots + x_{n+1}^2-x_{n+2}^2.\]
}
\begin{theorem}\label{thm:UCCbound}
    \thmUCCbound
\end{theorem}
A complete list of dimensions in which manifolds with the UCC property are known to exist can be found in  Table~\ref{tab:UCC_constructions}.
To date examples of flat manifolds with the UCC property in dimensions $n\equiv 3\bmod 4$ have proven harder to construct, with examples for dimensions $n\not\equiv 3\bmod{4}$ existing whenever $n\geq 10$. The simplest example of a manifold with the UCC property is $S^1$, which has the UCC property for the somewhat vacuous reason that there is only one commensurability class of non-compact arithmetic hyperbolic surfaces.
%The progression for the smallest possible examples was 543 -> 219 ->99 ->71-->59-->35.

It turns out that the number of commensurability classes of arithmetic hyperbolic manifolds containing cusp cross-sections with the UCC property tends to infinity with the dimension.
\newcommand{\thmManyUCCclasses}{Let $k\geq 1$ be an integer. Then for all sufficiently large $n$, there exist a collection $B_1,\dots, B_k$ of flat orientable $n$-manifolds such that each $B_i$ has the UCC property and for $i\neq j$ the manifolds $B_i$ and $B_j$ appear as cusp cross-sections in distinct commensurability classes of arithmetic hyperbolic manifolds.}
\begin{theorem}\label{thm:manyUCCclasses}
    \thmManyUCCclasses
\end{theorem}
We deduce Theorem~\ref{thm:manyUCCclasses} from the following result, which allows one to provide a quantitative lower bound on the meaning of sufficiently large dimension for each $k$. 
\begin{theorem}\label{thm:3mod4examples}
   Let $p\equiv 3\bmod 4$ be a prime and $n$ an integer such that
   $n\geq 2p^2+2p+44$.
   
   Then there exists a flat $n$-manifold with the UCC property appearing as a cusp cross-section in the commensurability class of arithmetic hyperbolic manifolds defined by the form
   \[q(x_1, \dots, x_{n+2})=\begin{cases}
   px_1^2+ px_2^2+x_3^2 +\dots + x_{n+1}^2-x_{n+2}^2 & n\not\equiv 2\bmod 4 \\
    px_1^2+ x_2^2 +\dots + x_{n+1}^2-x_{n+2}^2 & n\equiv 2\bmod 8\\
    px_1^2+ px_2^2+px_3^2+x_4^2 +\dots + x_{n+1}^2-x_{n+2}^2 & n\equiv 6\bmod 8.
    \end{cases}\]
\end{theorem}
% \begin{remark}
% In fact we have the slightly refined bounds.
%     \[
%     n\geq
%     \begin{cases}
%         2p^2+2p+8 &n\equiv 0\bmod 4\\
%         2p^2+2p+9 &n\equiv 1\bmod 4\\
%         p^2+p+10 & n\equiv 2 \bmod 4\\
%         2p^2+2p+71 &n\equiv 3\bmod 4.
%     \end{cases}
%     \]
% \end{remark}
\begin{table}[h]
    \centering
    \begin{tabular}{|c|c|c|} \hline
		Dimensions & Orientable  & Construction  \\ \hline
            $ 1 $ & Y & $S^1$  \\ \hline
    	$ 6$ & N &$C$ of Lemma~\ref{lem:C2}\\ \hline
        $ n\geq 10$, $n\equiv 2 \bmod 4 $ & Y & \begin{tabular}{c} $B_3$ of Lemma~\ref{lem:Bp} \\ 
         $C$ of Lemma~\ref{lem:C2} \end{tabular}\\ \hline
            $n\geq 12$, $n\equiv 0 \bmod 4$ & Y & $E$ of Lemma~\ref{lem:Emanifold} \\ \hline
            $n\geq 13$, $n\equiv 1 \bmod 4$ & Y & $S^1\times E$ of Lemma~\ref{lem:EtimesS1} \\ \hline
            
            $ n\geq 35, n\equiv 3 \bmod 4 $ & Y & $F$ of Lemma~\ref{lem:3mod4UCC} \\ \hline
	\end{tabular}
    \caption{A summary of the dimensions for which flat $n$-manifolds with the UCC property are known to exist} 
    \label{tab:UCC_constructions}
\end{table}

We say that two flat $n$-manifolds $B_1$ and $B_2$ form a \emph{non-arithmetic cusp pair} if there is no commensurability class of arithmetic hyperbolic manifolds that contains both $B_1$ and $B_2$ as cusp cross-sections. We say that the pair is orientable if both $B_1$ and $B_2$ are orientable. It was observed in \cite{McCoySell2024} that orientable non-arithmetic cusp pairs exist in all sufficiently large even dimensions ($n\geq 36$). Theorem~\ref{thm:manyUCCclasses} implies that orientable non-arithmetic cusp pairs exist in all sufficiently large dimensions. In particular, Theorem~\ref{thm:manyUCCclasses} gives examples of non-arithmetic cusp pairs with the stronger property that each element of the pair also has the UCC property. However, we are also able to construct non-arithmetic cusp pairs in smaller dimensions.
\begin{theorem}\label{thm:non-arith_pairs}
    For all $n\geq 24$ there exists a pair of flat $n$-manifolds forming an orientable non-arithmetic cusp pair.
\end{theorem}
Note that if a hyperbolic manifold has two cusps whose cross-sections form a non-arithmetic cusp pair, then it must be non-arithmetic.

A complete list of dimensions in which orientable non-arithmetic cusp pairs are known to exist can be found in  Table~\ref{tab:non-arith pairs}.
In fact, the results of \cite{McCoySell2024} imply that non-arithmetic cusp pairs actually have the stronger property that $B_1$ and $B_2$ can never appear as cusp cross-sections in the same commensurability class of quasi-arithmetic hyperbolic manifolds.
\subsection*{Proof of Theorems~\ref{thm:UCCbound} and~\ref{thm:non-arith_pairs}}
    Theorem~\ref{thm:UCCbound} follows from Lemma~\ref{lem:C2}, Lemma~\ref{lem:Emanifold}, Lemma~\ref{lem:EtimesS1} and  Lemma~\ref{lem:3mod4UCC}:
    \begin{enumerate}[(i)]
        \item Lemma~\ref{lem:C2} exhibits flat $n$-manifolds with the UCC property for all $n$ such that $n\geq 10$ and $n\equiv 2\bmod 4$;
        \item Lemma~\ref{lem:Emanifold} exhibits flat $n$-manifolds with the UCC property for all $n$ such that $n\geq 12$ and $n\equiv 0\bmod 4$;
        \item Lemma~\ref{lem:EtimesS1} exhibits flat $n$-manifolds with the UCC property for all $n$ such that $n\geq 13$ and $n\equiv 1\bmod 4$; and
        \item Lemma~\ref{lem:3mod4UCC} exhibits flat $n$-manifolds with the UCC property for all $n$ such that $n\geq 35$ and $n\equiv 3\bmod 4$.
    \end{enumerate}
   Theorem~\ref{thm:non-arith_pairs} follows from Lemma~\ref{lem:non-arithpairs-even_dims} and Lemma~\ref{lem:odd_dim_non-arith_pairs}. 
   Lemma~\ref{lem:non-arithpairs-even_dims} provides orientable non-arithmetic cusp pairs of dimension $n$ for all even $n\geq 20$ and Lemma~\ref{lem:odd_dim_non-arith_pairs} provides orientable non-arithmetic cusp pairs of dimension $n$ for all odd $n\geq 25$.

\subsection*{Remaining dimensions}
By examining the explicit tabulation of flat manifolds in small dimensions, one can check the existence of flat manifolds with various properties in dimensions $n\leq 6$. This analysis shows that there are no flat manifolds with the UCC property in dimensions $n=2,3,4$ and $5$ and that there are no orientable manifolds with the UCC property for $n=6$ \cite{GrimbertMcCoySell}. This leaves nine dimensions where the existence of flat $n$-manifolds with the UCC property is unknown.

\begin{question}
    For which of the dimensions $n=7,8,9,11,15,19,23,27$ and $31$ does there exist an (orientable) flat $n$-manifold with the UCC property?
\end{question}
The analysis of flat manifolds in low dimensions also shows that there are no non-arithmetic cusp pairs in dimensions $n\leq 6$ \cite{GrimbertMcCoySell}. Combined with the constructions of this paper, this leaves only eight dimensions in which the existence of orientable non-arithmetic cusp pairs is unknown.
\begin{question}
    For which of the dimensions $n=7,8,9,11,12,15,17$ and $23$ does there exist an orientable non-arithmetic cusp pair of dimension $n$?
\end{question}
\begin{table}[h]
    \centering
    \begin{tabular}{|c|c|c|} \hline
		Dimensions  & Construction \\ \hline
            
            $ n\geq 10$, $n\equiv 2 \bmod 4 $ &Lemma~\ref{lem:non-arithpairs-even_dims} \\ \hline
            $13$  &Lemma~\ref{lem:odd_dim_non-arith_pairs} \\ \hline
            $16$  & Lemma~\ref{lem:16dim_non_arith_pair} \\ \hline
            $19$  &Lemma~\ref{lem:odd_dim_non-arith_pairs} \\ \hline
            $n\geq 20$, $n\equiv 0 \bmod 4$ & Lemma~\ref{lem:non-arithpairs-even_dims} \\ \hline
            $n\geq 21$, $n\equiv 1 \bmod 4$ & Lemma~\ref{lem:odd_dim_non-arith_pairs} \\ \hline
            $n\geq 27$, $n\equiv 3 \bmod 4$ & Lemma~\ref{lem:odd_dim_non-arith_pairs} \\ \hline
	\end{tabular}
    \caption{A summary of the dimensions for which orientable non-arithmetic cusp pairs are known to exist.}
    \label{tab:non-arith pairs}
\end{table}

\subsection*{Acknowledgements}
DM is partially supported by NSERC grant RGPIN-2020-05491 and a Canada Research Chair. Both authors are grateful to the anonymous referee for their suggestions.

\section{Preliminaries}
\subsection{Rational quadratic forms}
Let $f: \Q^n \rightarrow \Q$ be a non-degenerate rational quadratic form. 
The orthogonal group of $f$ is the matrix group
\[ 
O(f;\Q)=\{A\in GL_n(\Q) \mid f(Av)=f(v)\, \forall v\in \Q^n\}.
\]
Two rational quadratic forms $f$ and $g$ are (rationally) equivalent if there exists a matrix $C\in GL_n(\Q)$ such that
\[
\text{$f(v)=g(Cv)$ for all $v\in \Q^n$.}
\]
We recall the rational equivalence class of a non-degenerate rational quadratic form $f$ is completely determined by the combination of its signature, discriminant $d(f)\in \Qsquare$ and Hasse-Witt invariant $\varepsilon_p(f)\in \{\pm1\}$ at each prime $p$ \cite{Serre_arithmetic}.  Since we will be making calculations with these invariants throughout this article, we briefly state their definitions and some basic properties.

Given $ a, b \in \mathbb{Q} $ and a prime $ p $, the Hilbert symbol $ (a,b)_p $ is defined to be $1$ if the equation $ ax^2 + by^2 = z^2 $ has a nonzero solution in $ \mathbb{Q}_p $, and $-1$ otherwise. The multiplicativity of the Hilbert symbol is well established, and the other properties below follow readily.
\begin{prop}
    Let $ a, b, c \in \Q $ and $ q $ any prime.  Then:
    \begin{itemize}
        \item $ (a,b)_q (a,c)_q = (a,bc)_q $,
        \item $ (a,1)_q = 1 $, and
        \item $ (a,a)_q = (a,-1)_q $.
    \end{itemize}
\end{prop}
Now we are ready to state the definitions of the discriminant and Hasse-Witt invariants.
\begin{defn}
    Let $ f $ be a diagonal quadratic form given by $ f(x) = a_1x_1^2 + \ldots + a_nx_n^2 $ with $ a_i \in \mathbb{Q} $ and $ q \in \mathbb{Z}_{>0} $ a prime.
    \begin{itemize}
        \item The \emph{discriminant} of $ f $ is $ d(f) = \prod a_i \in \Qsquare $.
        \item The \emph{Hasse-Witt invariant} of $ f $ at $ q $ is $ \varepsilon_q(f) = \prod_{i<j} (a_i, a_j)_q $.
    \end{itemize}
\end{defn}
Any quadratic form is equivalent to a diagonal form, so for any non-diagonal form we define these invariants to be those associated to an equivalent diagonal form. Note the discriminant is only well-defined up to multiplication by squares, and that each $ \varepsilon_q(f) \in \{ \pm 1 \} $. The formula above is more concise and easier to compute with, but morally, the Hasse-Witt invariant at $ p $ encodes the equivalence class of $ f $ interpreted over $ \mathbb{Q}_p $.

It follows easily from the definitions that these invariants are well behaved under direct sum of quadratic forms.
\begin{prop}\label{prop:sum_of_forms}
    Let $f$ and $g$ be two non-degenerate rational quadratic forms. Then
    \[
    d(f\oplus g)=d(f)d(g)\in \Qsquare
    \]
    and the Hasse-Witt invariants satisfy
    \[
    \varepsilon_p(f\oplus g)=\varepsilon_p(f)\varepsilon_p(g)\left(d(f),d(g)\right)_p,
    \]
    for every prime $p$.\qed
\end{prop}

\begin{defn}[Projective equivalence]
Two quadratic forms $ f $ and $ g $ over $ \mathbb{Q} $ are \emph{projectively equivalent (over $ \mathbb{Q} $)} if there exists a positive $ m \in \mathbb{Q}_{>0} $ such that $ mf $ and $ g $ are rationally equivalent.
\end{defn}
The projective equivalence class of a rational quadratic form can also be characterized in terms of the discriminant and Hasse-Witt invariants (see \cite[Proposition~5.4]{McCoySell2024}, for example).
\begin{prop}
\label{prop:projective_invariants}
    Let $f$ and $g$ be two positive definite quadratic forms of rank $n$ over $ \mathbb{Q} $. Then $f$ and $g$ are projectively equivalent if and only if one of the following conditions hold:
    \begin{enumerate}[(i)]
        \item $n\equiv 0 \bmod 4$, $d(f)=d(g)=d$ and $\varepsilon_p(f)=\varepsilon_p(g)$ for all primes $p$ such that $d$ is a square in $\Q_p$;
        \item $n\equiv 1 \bmod 4$ and $\varepsilon_p(f)=\varepsilon_p(g)$ for all primes $p$;
        \item $n\equiv 2 \bmod 4$, $d(f)=d(g)=d$ and $\varepsilon_p(f)=\varepsilon_p(g)$ for all primes $p$ such that $-d$ is a square in $\Q_p$ or
        \item $n\equiv 3 \bmod 4$ and $(d(f),-1)_p\varepsilon_p(f)=(d(g),-1)_p\varepsilon_p(g)$ for all primes $p$.
    \end{enumerate}
   \qed
\end{prop}
The following observation will be useful.
\begin{lemma}\label{lem:stable_equiv}
    Two rational quadratic forms $f$ and $f'$ are projectively equivalent if and only if the forms $f\oplus \langle 1,-1\rangle$ and $f'\oplus \langle 1,-1\rangle$ are projectively equivalent.
\end{lemma}
\begin{proof}
 Using the change of variables
\[
X^2 -Y^2=mx^2-my^2
\]
for
$X=\frac{1}{2}((1+m)x+(1-m)y)$ and $Y=\frac{1}{2}((1-m)x+(1+m)y)$, we observe that the diagonal forms $\langle m, -m\rangle$ and $\langle 1,-1 \rangle$ are rationally equivalent for all $m\neq 0$.

If $f\oplus \langle 1,-1\rangle$ and $f'\oplus \langle 1,-1\rangle$ are projectively equivalent, then there is $m\in \Q_{>0}$ such that $mf\oplus \langle m,-m\rangle$ and $f'\oplus \langle 1,-1\rangle$ are rationally equivalent. By the above observation it follows that $mf\oplus \langle 1,-1\rangle$ and $f'\oplus \langle 1,-1\rangle$ are rationally equivalent. By Witt cancellation, this means that $mf$ and $f'$ are rationally equivalent too, proving projective equivalence of $f$ and $f'$.

Conversely if $mf$ is rationally equivalent to $f'$, then $mf\oplus \langle 1,-1\rangle$ is rationally equivalent to both $mf\oplus \langle m,-m\rangle$ and $f'\oplus \langle 1,-1\rangle$. 
\end{proof}

\subsection{Holonomy forms}
Let $B$ be a compact flat $n$-manifold with fundamental group $\Gamma$. Then $\Gamma$ is a torsion-free group which we may identify as a discrete subgroup of $O(n)\ltimes \R^n$. The translation subgroup $T\leq \Gamma$ of elements acting by translations on $\R^n$ is a normal abelian subgroup. Moreover, the first Bieberbach theorem implies that $T$ is a maximal abelian subgroup of $\Gamma$ of finite index and that it is isomorphic to $\Z^n$. Thus identifying $T\cong \Z^n$ yields a short exact sequence
\begin{equation}
    1\rightarrow \Z^n \rightarrow \Gamma \rightarrow H \rightarrow 1,
\end{equation}
where the finite group $H$ is the \emph{holonomy group} of $B$. Letting $H$ act on $\Z^n$ by conjugation yields a representation:
\[
\rho_{\mathrm{hol}} : H \rightarrow GL_n(\Z).
\]
Since the translation subgroup is a maximal abelian subgroup of $\Gamma$, this representation is faithful.  Choosing a different identification of $T$ with $\Z^n$ has the effect of conjugating the representation $\rho_{\rm hol}$ by an element of $GL_n(\Z)$. Thus the conjugacy class of the representation $\rho_{\rm hol}$ is independent of this choice. It will also be useful to note that the manifold $B$ is orientable if and only if the image of $\rho_{\mathrm{hol}}$ is contained in $SL_n(\Z)$.
\begin{defn}
    We say that a (rational) \emph{holonomy representation} for $B$ is any rational representation
    \[
\rho : H \rightarrow GL_n(\Q)
\]
in the conjugacy class of representations defined by $\rho_{\mathrm{hol}}$.
A \emph{holonomy form} for $B$ is a positive definite rational quadratic form $f$ of rank $n$ such that $B$ has a holonomy representation 
\[
\rho:H \rightarrow O(f;\Q).
\]
\end{defn}
It is not hard to check that the set of all possible holonomy forms for a flat manifold $B$ is closed under projective equivalence.

In order to calculate the holonomy forms of a given flat manifold it will often be convenient to decompose the holonomy representation into subrepresentations (see \cite[Proposition 6.1]{McCoySell2024}, for example).
\begin{prop}\label{prop:form_decomposition}
    Let $\rho=\sigma_1 \oplus \dots \oplus \sigma_\ell$ be a direct sum of rational representations of a finite group $G$. If the image of $\rho$ is contained in $O(f;\Q)$, for a non-degenerate rational quadratic form $f$, then $f$ is equivalent to
    \[
    f\sim g_1 \oplus \dots \oplus g_\ell,
    \]
    where each $g_i$ is a quadratic form such that the image of $\sigma_i$ is contained in $O(g_i;\Q)$.
    \qed
\end{prop}

The toral extension construction of Auslander-Vasquez allows us to construct new flat manifolds from old \cite{Vasquez1970flat}.
\begin{prop}[{cf. \cite[Proposition~3.5]{McCoySell2024}}]\label{prop:toral_extension}
    Let $B$ be a compact flat manifold with holonomy representation $\rho: H\rightarrow GL_n(\Q)$. Then for any representation $\sigma: H\rightarrow GL_m(\Z)$, there is a compact flat $(n+m)$-manifold $\widetilde{B}$ with holonomy representation $\rho \oplus \sigma$.\qed
\end{prop}
Another useful fact is that $b_1$ of a flat manifold is encoded in the holonomy representation.
\begin{prop}[{\cite[Corollary~1.3]{Hiller1986flat}}]\label{prop:b1}
    Let $M$ be a compact flat manifold. Then $b_1(M)$ is equal to the number of trivial subrepresentations of the holonomy representation.\qed
\end{prop}
Many of the constructions in this paper will use flat manifolds of the form constructed in the following lemma.
\begin{lemma}\label{lem:Bp}
    For each odd prime $p$ and any dimension $n\geq p^2-1$ such that $n\equiv 0 \bmod p-1$, there exists a flat $n$-manifold $ B_p $ with $b_1(B_p)=0$ and holonomy group $H\cong (\Z/p\Z)^{2}$.
\end{lemma}
\begin{proof}
    The existence of such a manifold was shown for $n=p^2-1$ in \cite{Hiller1986flat}. Examples for $n>p^2-1$ can be constructed by taking toral extensions for some choice of non-trivial irreducible integral representations of $(\Z/p\Z)^{2}$. Since there exists a $p-1$-dimensional integral representation of $\Z/p\Z$, this gives examples in all the required dimensions.
\end{proof}

\subsection{Arithmetic hyperbolic manifolds}
For any rational quadratic form $q$ of signature $(n+1,1)$, there is a commensurability class of arithmetic hyperbolic $(n+1)$-manifolds obtained by considering subgroups of $O(q;\R)$ commensurable to $O(q;\Z)$. Moreover two rational quadratic forms $q_1$ and $q_2$ of signature $(n+1,1)$ define the same commensurability class of arithmetic hyperbolic manifolds if and only if $q_1$ and $q_2$ are projectively equivalent \cite[2.6]{Gromov}. The main result of \cite{McCoySell2024} determines precisely when a given flat manifold $B$ arises as a cusp cross-section in a commensurability class of arithmetic hyperbolic manifolds.
\begin{theorem}[{\cite[Theorem~1.1]{McCoySell2024}}]\label{thm:realization}
    Let $B$ be a compact flat $n$-manifold and let $q$ be a rational quadratic form of signature $(n+1,1)$. Then $B$ arises as a cusp cross-section in the commensurability class of arithmetic hyperbolic manifolds defined by $q$ if and only if $q$ is projectively equivalent to $f\oplus \langle 1,-1\rangle$ where $f$ is a holonomy form for $B$.\qed
\end{theorem}

Combined with Lemma~\ref{lem:stable_equiv}, we see that understanding the commensurability classes of arithmetic hyperbolic manifolds containing a given flat manifold $B$ as a cusp cross-section is equivalent to understanding its holonomy forms up to projective equivalence. In particular, showing that a flat manifold has the UCC property is equivalent to showing that it admits a unique holonomy form up to projective equivalence.

\begin{prop}\label{prop:UCCcharacterization}
    A flat manifold has the UCC property if and only if it admits a unique projective equivalence class of holonomy forms.
\end{prop}
\begin{proof}
    Suppose that $B$ is a compact flat manifold with the UCC property. Let $f$ and $f'$ be holonomy forms for $B$. By Theorem~\ref{thm:realization}, $B$ appears as a cusp cross-section in the commensurability classes of arithmetic hyperbolic manifolds defined by $q=f\oplus \langle 1,-1\rangle$ and  $q'=f'\oplus \langle 1,-1\rangle$. As $B$ has the UCC property, $q$ and $q'$ define the same commensurability class of arithmetic hyperbolic manifolds. This means that $q$ and $q'$ are projectively equivalent and, consequently, using Lemma~\ref{lem:stable_equiv} that $f$ and $f'$ are projectively equivalent. Conversely, suppose that $B$ admits a unique projective equivalence class of holonomy forms, say every holonomy form is projectively equivalent to a form $f$. Then Theorem~\ref{thm:realization} and Lemma~\ref{lem:stable_equiv} imply that if $B$ appears as a cusp cross-section in the commensurability class defined by a form $q$, then $q$ is projectively equivalent to $f\oplus \langle1,-1\rangle$. That is, $q$ and $f\oplus \langle1,-1\rangle$ define the same commensurability class of arithmetic hyperbolic manifolds.
\end{proof}

Showing that two flat manifolds form a non-arithmetic cusp pair is equivalent to showing that their sets of holonomy forms are disjoint.
\begin{prop}\label{prop:non_arith_pairs_characterization}
    A pair of flat manifolds form a non-arithmetic cusp pair if and only if they have no holonomy forms in common.
\end{prop}
\begin{proof}
Let $B_1$ and $B_2$ be a pair of compact flat manifolds. If $f$ is quadratic form that is a holonomy form for both $B_1$ and $B_2$, then Theorem~\ref{thm:realization} shows that $B_1$ and $B_2$ are both cusp cross-sections in the commensurability classes of arithmetic hyperbolic manifolds defined by $q=f\oplus \langle 1,-1\rangle$.

Conversely, if $B_1$ and $B_2$ are both cusp cross-sections in the commensurability classes of arithmetic hyperbolic manifolds defined by a form $q$, then Theorem~\ref{thm:realization} gives holonomy forms $f_1$ and $f_2$ for $B_1$ and $B_2$ respectively such that $f_1\oplus \langle 1,-1\rangle$ and $f_2\oplus \langle 1,-1\rangle$ are projectively equivalent to $q$. By Lemma~\ref{lem:stable_equiv} this means that $f_1$ and $f_2$ are themselves projectively equivalent. Since the set of holonomy forms is closed under projective equivalence, it follows that $f_1$ is a holonomy form for both $B_1$ and $B_2$.
\end{proof}

\subsection{Products of flat manifolds}
Let $M$ and $N$ be two compact flat manifolds of dimensions $m$ and $n$, respectively. Using the natural isomorphism $\pi_1(M\times N)\cong \pi_1(M)\times \pi_1(N)$, we see that the holonomy group of $M\times N$ is naturally a product $H_{M\times N}=H_M\times H_N$ and that there is a holonomy representation of the form
\begin{equation}\label{eq:product_holonomy}
\rho_{M\times N}=\rho_M \oplus \rho_N,
\end{equation}
where $\rho_M$, $\rho_N$ are the representations obtained by composing projections of $H_M\times H_N$ onto $H_M$ and $H_N$ with the holonomy representations for $M$ and $N$. Applying Proposition~\ref{prop:form_decomposition} to the holonomy representation in \eqref{eq:product_holonomy}, we obtain the following.
\begin{lemma}\label{lem:product_hol1}
    Let $M$ and $N$ be compact flat manifolds. Then any holonomy form for $M\times N$ is equivalent to $g\oplus h$, where $g$ is a holonomy form for $M$ and $h$ is a holonomy form for $N$. \qed
\end{lemma}
Furthermore, we will need the following refined statement.
\begin{lemma}\label{lem:product_hol_refined}
    Let $M$ and $N$ be compact flat manifolds, where $b_1(M)=0$. Let $f$ be a holonomy form for $M\times N$ such that a holonomy representation of the form given in \eqref{eq:product_holonomy} has image in $O(f;\Q)$. Then $f=g\oplus h$, where $g$ is a holonomy form for $M$ and $h$ is a holonomy form for $N$.
\end{lemma}
\begin{proof}
    As representations of $H_M\times H_N$, the only way that the representations $\rho_M$ and $\rho_N$ defined in \eqref{eq:product_holonomy} can have a subrepresentation in common is if the holonomy representations for $M$ and $N$ both contain a trivial subrepresentation. As $b_1(M)=0$, Proposition~\ref{prop:b1} implies that $\rho_M$ does not contain any trivial subrepresentations. Thus $\rho_M$ and $\rho_N$ do not share any isomorphic subrepresentations. Using Schur's lemma (cf. \cite[Lemma~6.3]{McCoySell2024}), this implies that if $\rho_M \oplus \rho_N$ has image in $O(f;\Q)$, then $f$ takes the form $f=g\oplus h$ where $\rho_M$ and $\rho_N$ have image in $O(g;\Q)$ and $O(h;\Q)$ respectively.
\end{proof}

\section{Manifolds with \texorpdfstring{$p$}{p}-group holonomy}
We now take a brief interlude to understand the representations of cyclic groups. Amongst other useful applications, this will provide examples of manifolds with the UCC property in dimensions $n\equiv 2\bmod 4$.

The cyclic group $\Z/n\Z$ admits a unique conjugacy class of faithful irreducible rational representations. We will denote this representation by 
\[
\sigma_n: \Z/n\Z \rightarrow GL_{r}(\Q),
\]
where the dimension satisfies $r=\varphi(n)$ with $\varphi$ denoting the Euler totient function. Over $\C$ the representation $\sigma_n$ decomposes as the sum of the $\varphi(n)$ distinct faithful 1-dimensional representations of $\Z/n\Z$.

\begin{prop}\label{prop:cyclic_rep_calc}
    Let $f$ be a positive definite quadratic form of rank $r=\varphi(n)$ such that $O(f;\Q)$ contains the image of the representation $\sigma_n$ for some $n>2$. Then
    \[
    d(f)=\begin{cases}
        q &\text{if $n=q^a$ or $n=2q^a$ for $a\geq 1$ and $q$ an odd prime}\\
        1 &\text{otherwise.}
    \end{cases}
    \]
    Suppose that there exists an integer $m\geq 1$ such that $-m$ is a square in both the cyclotomic field $\Q[\zeta_n]$ and the field $\Q_p$. Then over $\Q_p$ the form $f$ is equivalent to the diagonal form
    \[\langle \underbrace{1,\dots, 1}_{r/2}, \underbrace{-1,\dots, -1}_{r/2}\rangle.
    \]
   
\end{prop}
\begin{proof}
For $n$ an odd prime or $n=4$, the fact about the discriminant was established in \cite[Proposition~7.10 and Lemma~7.6]{McCoySell2024}. Thus suppose that $k$ is an integer dividing $n$ such that $k=p$ is either an odd prime or $k=4$. If we consider the restriction of $\sigma_n$ to $\Z/k\Z\leq \Z/n\Z$, then we see that the restricted representation decomposes as a sum of $\ell=r/\varphi(k)$ copies of $\sigma_k$. By Proposition~\ref{prop:form_decomposition}, this implies that $f$ is equivalent to a sum
\[
f\sim f_1 \oplus \dots \oplus f_\ell
\]
where $f_i$ is a form containing the image of $\sigma_k$. If $k=4$ or $\ell$ is even, this implies that $d(f)=1$. If $k=p$ is an odd prime and $\ell$ is odd, then $d(f)=p$. However, this latter case occurs if and only if $n$ is of the form $n=p^a$ or $n=2p^a$.

Now we turn to the Hasse-Witt calculations. If $-m$ is a square in $\Q[\zeta_n]$, then $\Q[\zeta_n]$ contains the imaginary quadratic field $\Q[\sqrt{-m}]$. Over $\Q[\sqrt{-m}]$ the representation $\sigma_n$ splits into two irreducible representations, which are necessarily complex. By \cite[Lemma~6.8]{McCoySell2024}, this implies that over $\Q[\sqrt{-m}]$ the form $f$ is equivalent to the diagonal form
\[
f\sim \langle \underbrace{1,\dots, 1}_{r/2}, \underbrace{-1,\dots, -1}_{r/2}\rangle.
\]
Consequently, we have the same equivalence over any extension of $\Q$ containing a copy of $\Q[\sqrt{-m}]$.
\end{proof}

The Galois theory of cyclotomic fields implies that for an odd prime $p$ the cyclotomic field $\Q[\zeta_p]$ contains a unique quadratic extension of $\Q$. For $p\equiv 3\bmod 4$, this quadratic subfield of $\Q[\zeta_p]$ is $\Q[\sqrt{-p}]$. This allows us to analyse the holonomy forms of manifolds whose holonomy group is a $p$-group for $p\equiv 3 \bmod 4$.

\begin{lemma}\label{lem:p-group_reps}
    Let $p$ be an odd prime and let $G$ be a finite $p$-group. Let $f$ be a non-degenerate  rational quadratic form of rank $n$ and let
    \[
    \rho:G\rightarrow O(f;\Q)
    \]
    be a representation not containing any trivial subrepresentations. Then
    \begin{enumerate}[(i)]
        \item\label{it:p-rep_rank} $n\equiv 0 \bmod p-1$;
        \item\label{it:p-rep_disc} $d(f)=p^{n/(p-1)}\in \Qsquare$; and
        \item\label{it:p-rep_HW} if $p\equiv 3 \bmod 4$ and $q$ is a prime such that $-p$ is a square in $\Q_q$, then
        \[\varepsilon_q(f)=\begin{cases}
        1 & \text{$q> 2$}\\
        (-1)^{\frac{n(n-2)}{8}}& \text{$q=2$.}
    \end{cases}\]
    \end{enumerate}
\end{lemma}
\begin{proof}
    We work first in the case that $\rho$ is irreducible. Since the image of $\rho$ is itself a non-trivial $p$-group, $\rho(G)$ has non-trivial centre. Thus there exists $z\in G$ such that $\rho(z)\neq I$ commutes with $\rho(g)$ for all $g\in G$. Moreover, we may assume that $\rho(z)$ has order $p$. Since $\rho(z)$ commutes with all elements of $\rho(G)$, we see that $\ker (\rho(z)-I)$ is a proper $G$-invariant subspace. Irreducibility of $\rho$ implies that $\ker (\rho(z)-I)$ is trivial. Thus, if we take $H$ to be the subgroup generated by $z$ and $\rho_H$ to be the restriction of $\rho$ to $H$, then we see that $\rho_H$ can be decomposed as
    \[
    \rho|_H\cong \rho_1 \oplus \dots \oplus \rho_\ell
    \]
    where each $\rho_i$ factors through $\sigma_p$, where $\sigma_p$ is the unique non-trivial irreducible rational representation of $\Z/p\Z$. By Proposition~\ref{prop:form_decomposition}, it follows that $f$ is equivalent to
    \[
    f\sim f_1 \oplus \dots \oplus f_\ell
    \]
    where each $f_i$ is a quadratic form such that $O(f_i;\Q)$ contains the image of $\sigma_p$.

    If $\rho$ is reducible, then Proposition~\ref{prop:form_decomposition} allows us to apply the argument of the preceding paragraph to each of the irreducible subrepresentations of $\rho$. This shows that for any $\rho$ not containing any trivial subrepresentations, the form $f$ is equivalent to
    \[
    f\sim f_1 \oplus \dots \oplus f_\ell
    \]
    where each $f_i$ is a quadratic form such that $O(f_i;\Q)$ contains the image of $\sigma_p$.

    Since $\sigma_p$ is $(p-1)$-dimensional, it follows that $n=\ell(p-1)$, establishing \eqref{it:p-rep_rank}. By Proposition~\ref{prop:cyclic_rep_calc} and Proposition~\ref{prop:sum_of_forms}, we have that $d(f)=p^{\ell}=p^{n/(p-1)}$. This gives \eqref{it:p-rep_disc}. Finally, suppose that $p\equiv 3\bmod 4$ and that $q$ is a prime such that $-p$ is a square in $\Q_q$. Since $p\equiv 3\bmod 4$, we have that $-p$ is a square in $\Q[\zeta_p]$. Thus Proposition~\ref{prop:cyclic_rep_calc} implies that over $\Q_q$ each $f_i$ is equivalent to the diagonal form 
    \[
f_i\sim \langle \underbrace{1,\dots, 1}_{\frac{p-1}{2}}, \underbrace{-1,\dots, -1}_{\frac{p-1}{2}}\rangle.
\]
Thus $f$ being equivalent to the sum of the $f_i$ is equivalent to 
 \[
f\sim \langle \underbrace{1,\dots, 1}_{\frac{n}{2}}, \underbrace{-1,\dots, -1}_{\frac{n}{2}}\rangle.
\]
From this, we calculate that
\[\varepsilon_q(f)=(-1,-1)_q^{\frac{n(n-2)}{8}}=\begin{cases}
        1 & \text{$q> 2$}\\
        (-1)^{\frac{n(n-2)}{8}}& \text{$q=2$,}
    \end{cases}\]
as required for \eqref{it:p-rep_HW}.
\end{proof}
For flat manifolds, this translates immediately into the following conditions.
\begin{prop}\label{prop:modpholonomy}
    Let $p\equiv 3\bmod 4$ be a prime. Let $B$ be a flat $n$-manifold such that $b_1(B)=0$ and its holonomy group is a $p$-group. Then $n$ is even and for any holonomy form $f$ for $B$ we have
    \[
    d(f)=\begin{cases}
        1 &n\equiv 0 \bmod 4\\
        p &n\equiv 2\bmod 4
    \end{cases}
    \]
    and for all primes $q$ such that $-p$ is a square in $\Q_q$ we have 
    \[\varepsilon_q(f)=\begin{cases}
        1 & \text{$q> 2$}\\
        (-1)^{\frac{n(n-2)}{8}}& \text{$q=2$.}
    \end{cases}\]
\end{prop}
\begin{proof}
        Let $\rho$ be the holonomy representation for $B$. Since we are supposing that $b_1(B)=0$, Proposition~\ref{prop:b1} implies that $\rho$ contains no trivial subrepresentations. The remaining statements now follow immediately from Lemma~\ref{lem:p-group_reps} using the fact that $p-1\equiv 2\bmod 4$.
\end{proof}

Consequently, we obtain a natural class of manifolds with the UCC property in dimensions $n\equiv 2\bmod 4$. In particular, some of the manifolds of the form given by Lemma~\ref{lem:Bp} turn out to have the UCC property.
\begin{theorem}
\label{thm:p3mod4UCCexamples}
Let $p\equiv 3\bmod 4$ be a prime. Let $B$ be a flat $n$-manifold such that $b_1(B)=0$ and its holonomy group is a $p$-group. If $n\equiv 2\bmod 4$, then $B$ has the UCC property and appears as a cusp cross-section in the commensurability class of arithmetic hyperbolic manifolds defined by the quadratic form
    \[
    q(x_1,\dots,x_{n+2})=\begin{cases}
    px_1^2 +x_2^2 +\dots +x_{n+1}^2-x_{n+2}^2 &n\equiv 2\bmod 8\\
        px_1^2 +px_2^2 +px_3^2+x_4^2+\dots +x_{n+1}^2-x_{n+2}^2 &n\equiv 6\bmod 8
    \end{cases}
    \]
\end{theorem}
\begin{proof}
Since $b_1(B)=0$, the holonomy representation for $B$ does not contain any trivial subrepresentations. Thus Lemma~\ref{lem:p-group_reps} shows that any holonomy form for $B$ satisfies $d(f)=p$ and for any prime $q$ such that $-p$ is a square in $\Q_q$ we have
    \[\varepsilon_q(f)=\begin{cases}
        1 & \text{$q\neq 2$ or $n\equiv 2\bmod 8$}\\
        -1& \text{$q=2$ and $n\equiv 6\bmod 8$.}
    \end{cases}.\]
If $n\equiv 2\bmod 4$, then Proposition~\ref{prop:projective_invariants} shows that $B$ admits a unique projective equivalence class of holonomy form. By Proposition~\ref{prop:UCCcharacterization}, this implies that $B$ has the UCC property when $n\equiv 2\bmod 4$. One can easily verify that the form  
\[
    f(x_1,\dots,x_{n})=\begin{cases}
    px_1^2 +x_2^2 +\dots +x_{n}^2 &n\equiv 2\bmod 8\\
        px_1^2 +px_2^2 +px_3^2+x_4^2+\dots +x_{n}^2 &n\equiv 6\bmod 8
    \end{cases}
    \]
    is a representative of the unique projective class of holonomy form.
\end{proof}

\section{The UCC property in dimensions \texorpdfstring{$n\equiv 2\bmod 4$}{n=2 mod 4}}
In this section, we construct a family of manifolds with the UCC property in dimensions $n\equiv 2\bmod 4$. This includes an example of a (non-orientable) 6-dimensional manifold with the UCC property. First, we take toral extensions of the Hantzsche–Wendt manifold in order to obtain various flat manifolds with $b_1(B)=0$ and $H\cong (\Z/2\Z)^2$.
\begin{lemma}\label{lem:HW_extension}
    If $n\geq 4$, then there exists a non-orientable flat $n$-manifold $B$ with $b_1(B)=0$ and $H\cong (\Z/2\Z)^2$. If $n=3$ or $n\geq 5$, then there exists an orientable flat $n$-manifold $B$ with $b_1(B)=0$ and $H\cong (\Z/2\Z)^2$.
\end{lemma}
\begin{proof}
    Let $B_{HW}$ denote the Hantzsche–Wendt manifold, which is the unique orientable 3-dimensional flat manifold with $b_1(B_{HW})=0$ and $H\cong (\Z/2\Z)^2$. The holonomy representation of $B_{HW}$ can be decomposed as $\rho \cong \rho_1 \oplus \rho_2 \oplus \rho_3$, where the $\rho_i$ are the three distinct non-trivial irreducible representations of  $(\Z /2)^2$. By performing toral extensions as in Proposition~\ref{prop:toral_extension} we can obtain, for any triple of integers $a_1, a_2, a_3\geq 1$, a flat manifold $B$ with $H\cong (\Z/2\Z)^2$, $b_1(B)=0$ and holonomy representation $\rho_1^{\oplus a_1} \oplus \rho_2^{\oplus a_2} \oplus \rho_3^{\oplus a_3}$. This manifold has dimension $a_1+a_2+a_3$. Since the non-identity elements of $(\Z/2\Z)^2$ have determinants $(-1)^{a_1+a_2}$, $(-1)^{a_1+a_3}$ and $(-1)^{a_2+a_3}$ under this holonomy representation, we see that $B$ is orientable if and only if $a_1, a_2$ and $a_3$ all have the same parity.
    For $n\geq 5$, taking
    \[
    (a_1,a_2,a_3)=\begin{cases}
        (n-2,1,1) &\text{$n$ odd}\\
        (n-4,2,2) &\text{$n$ even},
    \end{cases}
    \]
    yields an orientable manifold.
    For $n\geq 4$, taking $(a_1,a_2,a_3)=(n-3,2,1)$ yields a non-orientable manifold.
\end{proof}

\begin{lemma}\label{lem:C2}
    For every even $n\geq 6$, there exists a flat $n$-manifold $C$ such that
    \begin{enumerate}[(i)]
    \item\label{it:C2_holforms} Every holonomy form $f$ for $C$ satisfies $d(f)=1$ and $\varepsilon_q(f)=1$ for all primes such that $-1$ is a square in $\Q_q$.
    \item\label{it:C2_orientability} $C$ is orientable if $n\geq 10$ and non-orientable if $n=6,8$.
    % \item\label{it:C2_cover} The holonomy group $H$ of $C_2$ admits a double cover $\wt{C_2}\rightarrow C_2$ induced by a non-trivial homomorphism $H\rightarrow \Z/2\Z$ such that $b_1(\wt{C_2})=0$ and $\wt{C_2}$ is orientable.
      \item\label{it:C2hasUCCproperty} If $n\equiv 2\bmod 4$, then $C$ has the UCC property and appears as a cusp cross-section in the commensurability class defined by
      \[
      q(x_1, \dots, x_{n+2})=x_1^2 +\dots + x_{n+1}^2-x_{n+2}^2.
      \]
    \end{enumerate}
\end{lemma}
\begin{proof}
    Let $B$ be a flat manifold of dimension $k\geq 3$ with $b_1(B)=0$ and holonomy group $H_B\cong (\Z/2\Z)^2$ as constructed in Lemma~\ref{lem:HW_extension}. We choose a surjection $H_B\rightarrow \Z/2\Z$ such that we have a double cover $\wt{B}\rightarrow B$ where $\wt{B}$ is orientable. Let $\alpha: \wt{B} \rightarrow \wt{B}$ be the non-trivial deck transformation of this cover. We define an action of $\Z/4\Z$ on $\wt{B}\times \wt{B}$ where the generator acts by the isometry
    \begin{align*}
        \beta:\wt{B}\times \wt{B}&\rightarrow\wt{B}\times \wt{B}\\
        (x,y)&\mapsto (\alpha(y),x).
    \end{align*}
    Note that $\beta$ has no fixed points since $\beta^2(x,y)=(\alpha(x),\alpha(y))$ has no fixed points. Let $C$ be the quotient of $\wt{B}\times \wt{B}$ by this action of $\Z/4\Z$. We see that $C$ is a $2k$-dimensional flat manifold. Note that $C$ is orientable if and only if $\beta$ is orientation preserving. If $k$ is odd, then $\beta$ is orientation preserving if and only if $\alpha$ is orientation reversing. Thus for $k$ odd, we see that $C$ is orientable if and only if $B$ is non-orientable. By Lemma~\ref{lem:HW_extension}, this is possible for all odd $k\geq 5$. If $k$ is even, then $\beta$ is orientation preserving if and only if $\alpha$ is orientation preserving. Thus for $k$ even, we see that $C$ is orientable if and only if $B$ is orientable. By Lemma~\ref{lem:HW_extension}, this is possible for all even $k\geq 6$. Thus we see that $B$ can be chosen so that $C$ is orientable if $k\geq 5$ and non-orientable for $k=3,4$. This verifies \eqref{it:C2_orientability}.

    Next we calculate the holonomy forms for $C$. Let $\rho$ be a holonomy representation for $B$. Since the cover $\wt{B}\rightarrow B$ is induced by a surjection $H_B\rightarrow \Z/2\Z$, we may consider $H_{\wt{B}}$ as a subgroup of $H_{B}$. By construction we see that there is a holonomy representation $\sigma$ for $C$ whose image is generated by matrices of the form
    \begin{align}\label{eq:B_hol1}
    \begin{pmatrix}
        \rho(h_1) &0\\
        0& \rho(h_2)
    \end{pmatrix}&\quad \text{for $h_1,h_2\in H_{\wt{B}}$ and}\\
    \begin{pmatrix}
        0 &\rho(h_1)\\
        \rho(h_2)&0
    \end{pmatrix}&\quad \text{for $h_1\in H_B\setminus H_{\wt{B}}$ and $h_2\in H_{\wt{B}}$.}\label{eq:B_hol2}
    \end{align}
    Here the matrices of \eqref{eq:B_hol1} are those of the holonomy representation of $\wt{B}\times \wt{B}$ and those of \eqref{eq:B_hol2} arise from the action of $\beta$.

    Let $\rho_1, \rho_2, \rho_3$ be the three non-trivial irreducible representations of $H_B\cong (\Z/2\Z)^2$. The holonomy representation for $\rho$ decomposes as a sum of copies of these three representations. This implies that the holonomy representation $\sigma$ can be decomposed as a sum of representations $\sigma_i$ whose images are generated by
    \[\text{
    $\begin{pmatrix}
        \rho_i(h_1) &0\\
        0& \rho_i(h_2)
    \end{pmatrix}$ for $h_1,h_2\in H_{\wt{B}}$
    }\]
    and 
    \[\text{
    $\begin{pmatrix}
        0 &\rho_i(h_1)\\
        \rho_i(h_2)&0
    \end{pmatrix}$ for $h_1\in H_B\setminus H_{\wt{B}}$ and $h_2\in H_{\wt{B}}$.}
    \]
    However, for all $i=1,2,3$, it is easy to verify that the representation $\sigma_i$ contains the matrix $\begin{pmatrix}
        0 &-1\\
        1 & 0
    \end{pmatrix}$ in its image. Thus we see that if $\sigma_i$ has image in $O(g;\Q)$, then $g$ is a diagonal form $g=\langle a, a\rangle$. By Proposition~\ref{prop:form_decomposition}, this implies that any holonomy form $f$ for $C$ is equivalent to a diagonal form
    \[
    f\sim \langle a_1, \dots, a_k, a_1, \dots, a_k \rangle.
    \]
    for $a_1, \dots, a_k \in \Q_{>0}$. That is, $f\sim g\oplus g$ for some form $g$. From this we calculate that
    \[
    \text{$d(f)=1$ and $\varepsilon_q(f)=\varepsilon_q(g)^2 (d(g), d(g))_q=(d(g),-1)_q$ for all $q$.}
    \]
    In particular, $\varepsilon_q(f)=1$ whenever $-1$ is a square in $\Q_q$. This proves \eqref{it:C2_holforms}.

    Finally, we verify \eqref{it:C2hasUCCproperty}. Suppose that $n\equiv 2 \bmod 4$. In this case, Proposition~\ref{prop:projective_invariants} combined with \eqref{it:C2_holforms}  implies that the holonomy forms for $C$ all lie in a single projective equivalence class. By Proposition~\ref{prop:UCCcharacterization}, this implies that $C$ has the UCC property. One can easily see that the form
     \[f(x_1, \dots, x_{n})=x_1^2 +\dots + x_{n}^2\]
     is a representative of the projective equivalence class of holonomy forms for $C$. By Theorem~\ref{thm:realization} it follows that $C$ appears as a cusp cross-section in the commensurability class of arithmetic hyperbolic manifolds defined by
     \[q(x_1, \dots, x_{n+2})=x_1^2 +\dots + x_{n+1}^2-x_{n+2}^2,\]
     as required.
\end{proof}

\section{The UCC property in dimensions \texorpdfstring{$n\equiv 0,1\bmod 4$}{n=0,1 mod 4}}
In this section, we first construct a family of flat manifolds with the UCC property in all dimensions $n\equiv 0\bmod 4$ and $n\geq 12$. In addition to being examples forming part of Theorem~\ref{thm:UCCbound}, these manifolds will be important building blocks in future constructions. The construction mimics that of Lemma~\ref{lem:C2}. However, rather than finding an action of $\Z/4\Z$ on $\wt{B}\times \wt{B}$, we will construct an action of the quaternion group $Q_8$ on $\wt{B}^{\times 4}$.
\begin{lemma}\label{lem:Emanifold}
    For all $n\geq 12$ and $n\equiv 0 \bmod 4$, there is an orientable flat $n$-manifold $E$ such that
    \begin{enumerate}[(i)]
        \item\label{it:E_hol_forms} every holonomy form $f$ of $E$ satisfies $d(f)=1$ and $\varepsilon_p(f)=1$ for all primes $p$;
        \item\label{it:E_cover} there is a connected double cover $\wt{E}\rightarrow E$ such that $b_1(\wt{E})=0$; and
        \item\label{it:EhasUCCproperty} $E$ has the UCC property and appears as  a cusp cross-section in the commensurability class of arithmetic hyperbolic manifold defined by
     \[q(x_1, \dots, x_{n+2})=x_1^2 +\dots + x_{n+1}^2-x_{n+2}^2.\]
    \end{enumerate}
\end{lemma}

\begin{proof}
    Let $B$ be a flat manifold of dimension $k\geq 3$ with $b_1(B)=0$ and holonomy group $H_B\cong (\Z/2\Z)^2$ as constructed in Lemma~\ref{lem:HW_extension}. We choose a surjection $H_B\rightarrow \Z/2\Z$ which induces a double cover $\wt{B}\rightarrow B$ where $\wt{B}$ is orientable. Let $\alpha: \wt{B}\rightarrow \wt{B}$ be the non-trivial deck transformation of this cover. The quaternion group 
\[Q_8=\langle i,j | i^2=j^2, i^4=e, iji=j\rangle\]
acts on $\wt{B}^{\times 4}$ by isometries
\[
\text{$i(x,y,z,w)= (y,\alpha(x),w,\alpha(z))$ and $j(x,y,z,w)=(z,\alpha(w),\alpha(x),y)$.}
\]
Observe that the central element $i^2$ acts by the diagonal action
\[
i^2(x,y,z,w)= (\alpha(x),\alpha(y),\alpha(z),\alpha(w))
\]
on $\wt{B}^{\times 4}$ which has no fixed points since $\alpha$ has no fixed points on $\wt{B}$.
Since every non-central element $g\in Q_8$ satisfies $g^2=i^2$, this implies that every non-trivial element of $Q_8$ acts on $\wt{B}^{\times 4}$ without fixed points. Thus the quotient $E=\wt{B}^{\times 4}/Q_8$ is a $4k$-dimensional flat manifold. Moreover $\wt{B}$ is orientable and one can check that both $i$ and $j$ are orientation-preserving (irrespective of whether or not $B$ is orientable). This implies that $E$ is orientable.

Next we calculate the holonomy forms for $C$. Let $\rho$ be a holonomy representation for $B$. We consider $H_{\wt{B}}$ as a subgroup of $H_{B}$. By construction we see that there is a holonomy representation $\sigma$ for $C$ whose image is generated by matrices of the form
\begin{align}\label{eq:holE1}
    \begin{pmatrix}
        \rho(h_1) &0&0&0\\
        0& \rho(h_2)&0&0\\
        0& 0&\rho(h_3)&0\\
        0& 0& 0&\rho(h_4)
    \end{pmatrix}&\quad \text{for $h_1,h_2,h_3, h_4\in H_{\wt{B}}$;}\\
    \label{eq:holE2}\begin{pmatrix}
        0&\rho(h_1) &0&0\\
        \rho(h_2)&0&0&0\\
        0& 0&0&\rho(h_3)\\
        0& 0&\rho(h_4) &0
    \end{pmatrix}& \quad\text{for $h_1,h_3\in H_B\setminus H_{\wt{B}}$ and $h_2,h_4\in H_{\wt{B}};$ and}\\
    \label{eq:holE3}\begin{pmatrix}
        0&0 &\rho(h_3)&0\\
        0&0&0&\rho(h_4)\\
        \rho(h_1)& 0&0&0\\
        0& \rho(h_2) &0&0
    \end{pmatrix}&\quad \text{for $h_1,h_4\in H_B\setminus H_{\wt{B}}$ and $h_2,h_3\in H_{\wt{B}}$}.
    \end{align}
    The matrices of \eqref{eq:holE1} are precisely those of the holonomy representation of $\wt{B}^{\times 4}$ and the matrices of \eqref{eq:holE2} and \eqref{eq:holE3} are those coming from the action of $i$ and $j$.

    Let $\rho_1, \rho_2, \rho_3$ be the three non-trivial irreducible representations of $H_B\cong (\Z/2\Z)^2$. The holonomy representation for $\rho$ decomposes as a sum of copies of these three representation. This implies that the holonomy representation $\sigma$ can be decomposed as a sum of representations $\sigma_i$ whose images are generated by
    \begin{align*}
    \begin{pmatrix}
        \rho_i(h_1) &0&0&0\\
        0& \rho_i(h_2)&0&0\\
        0& 0&\rho_i(h_3)&0\\
        0& 0& 0&\rho_i(h_4)
    \end{pmatrix}&\quad \text{for $h_1,h_2,h_3, h_4\in H_{\wt{B}}$;}\\
    \begin{pmatrix}
        0&\rho_i(h_1) &0&0\\
        \rho_i(h_2)&0&0&0\\
        0& 0&0&\rho_i(h_3)\\
        0& 0&\rho_i(h_4) &0
    \end{pmatrix}& \quad\text{for $h_1,h_3\in H_B\setminus H_{\wt{B}}$ and $h_2,h_4\in H_{\wt{B}};$ and}\\
    \begin{pmatrix}
        0&0 &\rho_i(h_3)&0\\
        0&0&0&\rho_i(h_4)\\
        \rho_i(h_1)& 0&0&0\\
        0& \rho_i(h_2) &0&0
    \end{pmatrix}&\quad \text{for $h_1,h_4\in H_B\setminus H_{\wt{B}}$ and $h_2,h_3\in H_{\wt{B}}$}.
    \end{align*}
    However, for all $i=1,2,3$, it is easy to verify that the image of the representation $\sigma_i$ contains the matrices
    \[
    \begin{pmatrix}
        0&-1 &0&0\\
        1&0&0&0\\
        0& 0&0&-1\\
        0& 0&1 &0
    \end{pmatrix}\quad\text{and}\quad
    \begin{pmatrix}
        0&0 &1&0\\
        0&0&0&-1\\
        -1& 0&0&0\\
        0& 1 &0&0
    \end{pmatrix}.
    \]
    From these one can verify that if $\sigma_i$ has image in $O(g;\Q)$, then $g$ is a diagonal form $g=\langle a, a,a,a\rangle$. Thus we see that any holonomy form $f$ for $C$ is equivalent to
    \[
    f\sim h\oplus h\oplus h\oplus h
    \]
     for some rational form $h$. From this we calculate that
    \[
    \text{$d(f)=1$ and $\varepsilon_q(f)=1$ for all $q$,}
    \]
    as required for \eqref{it:E_hol_forms}.
    Next we construct the cover $\wt{E}\rightarrow E $ required for \eqref{it:E_cover}. We take the quotient of $\wt{B}^{\times 4}$ by the $\Z/4\Z$-action generated by $i$. This is clearly a double cover of $E$. The holonomy representation for $\wt{E}$ is generated by the matrices of the form \eqref{eq:holE1} and \eqref{eq:holE2}. This representation can be checked to contain no trivial representations. For example, one sees that upon restriction to the holonomy group of $\wt{E}$ each of the representations $\sigma_i$ decomposes into two 2-dimensional irreducible rational representations. By Proposition~\ref{prop:b1}, this implies that $b_1(\wt{E})=0$, as required.

    Finally, we verify \eqref{it:EhasUCCproperty}. Proposition~\ref{prop:projective_invariants} combined with \eqref{it:E_hol_forms}  implies that the holonomy forms for $E$ all lie in a single projective equivalence class. By Proposition~\ref{prop:UCCcharacterization}, this implies that $E$ has the UCC property. One can easily see that the form
     \[f(x_1, \dots, x_{n})=x_1^2 +\dots + x_{n}^2\]
     is a representative of the projective equivalence class of holonomy forms for $E$. By Theorem~\ref{thm:realization} it follows that $E$ appears as a cusp cross-section in the commensurability class of arithmetic hyperbolic manifolds defined by
     \[q(x_1, \dots, x_{n+2})=x_1^2 +\dots + x_{n+1}^2-x_{n+2}^2,\]
     as required.
\end{proof}
Taking the product $E\times S^1$ yields flat manifolds with the UCC property in dimensions $n\equiv 1 \bmod 4$. More generally, it is not hard to verify that if $M$ is any compact flat manifold with the UCC property, then a manifold of the form $E\times M$ also has the UCC property.
\begin{lemma}\label{lem:EtimesS1}
    For all $n$ such that $n\geq 13$ and $n\equiv 1 \bmod 4$, there exists an orientable flat $n$-manifold with the UCC property that appears as a cusp cross-section in the commensurability class of arithmetic hyperbolic manifold defined by
     \[q(x_1, \dots, x_{n+2})=x_1^2 +\dots + x_{n+1}^2-x_{n+2}^2.\]
     %such that every holonomy form satisfies $\varepsilon_p(f)=1$ for all primes $p$.
\end{lemma}
\begin{proof}
Let $E$ be a flat manifold as constructed in Lemma~\ref{lem:Emanifold}. By Lemma~\ref{lem:product_hol1}, any holonomy form $f$ for $S^1\times E$ is equivalent to 
    \[
    f\sim \langle a \rangle \oplus g,
    \]
    where $a\in \Q_{>0}$ and $g$ is a holonomy form for $E$. Using the properties of the Hasse-Witt invariant as in Proposition~\ref{prop:sum_of_forms}, we find that
    \[
    \varepsilon_p(f)=(a,d(g))_p \varepsilon_p(g)=1,
    \]
    for all primes $p$. By Proposition~\ref{prop:projective_invariants} this means that the holonomy forms for $S^1\times E$ all lie in a single projective equivalence class. By Proposition~\ref{prop:UCCcharacterization}, this implies that $S^1\times E$ has the UCC property. One can easily see that the form
     \[f(x_1, \dots, x_{n})=x_1^2 +\dots + x_{n}^2\]
     is a representative of the projective equivalence class of holonomy forms for $S^1\times E$. By Theorem~\ref{thm:realization}, it follows that $S^1\times E$ appears as a cusp cross-section in the commensurability class of arithmetic hyperbolic manifolds defined by
     \[q(x_1, \dots, x_{n+2})=x_1^2 +\dots + x_{n+1}^2-x_{n+2}^2,\]
     as required.
\end{proof}

\section{The UCC property in dimensions \texorpdfstring{$n\equiv 3 \bmod 4$}{n=3 mod 4}}
We now embark on the construction of flat manifolds with the UCC property in dimensions $n\equiv 3\bmod 4$. These will be the final manifolds required to complete the proof of Theorem~\ref{thm:UCCbound}. First we construct some useful manifolds with holonomy group $(\Z/3\Z)^3$.

\begin{lemma}\label{lem:C3}
    For any even $n\geq 10$, there exists an orientable flat $n$-manifold $C$ with a cyclic 3-fold cover $\wt{C}\rightarrow C$ such that $b_1(\wt{C})=0$ and the holonomy group of $\wt{C}$ is $(\Z/3\Z)^2$.
\end{lemma}
\begin{proof}
    We write down a description for the manifold $C$ for $n=10$. This is modelled on the manifolds constructed by Hiller-Sah with holonomy group $(\Z/p\Z)^{2}$ \cite{Hiller1986flat}. Let $\alpha, \beta,\gamma$ be a set of generators for $(\Z/3\Z)^3$. Consider the matrix of order three $A=\begin{pmatrix}
        0 & -1\\
        1& -1
    \end{pmatrix}$ and the vector $\mu=\begin{pmatrix}
        \frac23 \\ \frac13 
    \end{pmatrix}$.
    We define an action of $(\Z/3\Z)^3$ on the 10-dimensional torus $T^{10}=\R^{10}/\Z^{10}$ by setting
    \[
    \alpha(v)=\begin{pmatrix}
        I &0 &0 &0 &0\\
        0 &A &0 &0 &0\\
        0 &0 &A &0 &0\\
        0 &0 &0 &A &0\\
        0 &0 &0 &0 &A
    \end{pmatrix}v +
    \begin{pmatrix}
        \mu\\
        \mu \\
        0 \\
        0 \\
        0
    \end{pmatrix},
    \]
    \[
    \beta(v)=\begin{pmatrix}
        A &0 &0 &0 &0\\
        0 &A &0 &0 &0\\
        0 &0 &I &0 &0\\
        0 &0 &0 &A^2 &0\\
        0 &0 &0 &0 &A
    \end{pmatrix}v +
    \begin{pmatrix}
        0\\
        0 \\
        \mu \\
        \mu \\
        0
    \end{pmatrix}
    \]
    and
    \[
    \gamma(v)=\begin{pmatrix}
        I &0 &0 &0 &0\\
        0 &A &0 &0 &0\\
        0 &0 &I &0 &0\\
        0 &0 &0 &I &0\\
        0 &0 &0 &0 &I
    \end{pmatrix}v +
    \begin{pmatrix}
        0\\
        0 \\
        0 \\
        0 \\
        \mu
    \end{pmatrix}
    \]
    One can check that this is in fact a well-defined action of $(\Z/3\Z)^{3}$: the actions of $\alpha$, $\beta$ and $\gamma$ all have order three and the fact that $A\mu-\mu \in \Z^2$ implies that the actions of $\alpha$, $\beta$ and $\gamma$ all commute.
    
    Furthermore, by direct calculation, one can verify that this acts without fixed points. Thus the quotient is a flat manifold with holonomy group $(\Z/3\Z)^3$. We obtain a $\Z/3\Z$-cover $\wt{C}$ of $C$ by taking the quotient of $T^{10}$ by the copy of $(\Z/3\Z)^{2}$ subgroup generated by $\alpha$ and $\beta$. By definition of $\alpha$ and $\beta$, the holonomy representation of $\wt{C}$ contains no trivial subrepresentations. By Proposition \ref{prop:b1}, this means $b_1(\wt{C})=0$. For the manifolds of dimension $n\geq 12$, we take toral extensions of $C$ by 2-dimensional rational representations of $(\Z/3\Z)^3$ that are non-trivial on the subgroup generated by $\alpha$ and $\beta$. 
\end{proof}
Next we construct an isometry of order three on the Hantzsche-Wendt manifold. The 3-dimensional Hantzsche-Wendt manifold $B_{HW}$ can be constructed as the quotient of the 3-dimensional torus $T^3= \R^3/\Z^3$, by the action of $(\Z/2\Z)^2$ generated by $\gamma, \delta$, which act by
\[
\gamma(v)=\begin{pmatrix}
    -1 & 0 & 0\\
    0 & -1 & 0\\
    0 & 0 & 1\\
\end{pmatrix}v +\begin{pmatrix}
    0 \\
    \frac12 \\
    \frac12 \\
\end{pmatrix}
\]
and \[
\delta(v)=\begin{pmatrix}
    1 & 0 & 0\\
    0 & -1 & 0\\
    0 & 0 & -1\\
\end{pmatrix}v +\begin{pmatrix}
\frac12 \\
    0 \\
    \frac12 \\
\end{pmatrix}.
\]
From this description, we see that the isometry of $T^3= \R^3/\Z^3$ that cyclically permutes the coordinates descends to an isometry $\tau: B_{HW}\rightarrow B_{HW}$. 
\begin{lemma}\label{lem:3mod4UCC}
    For every $n$ such that $n\geq 35$ and $n\equiv 3\bmod 4$, there exists an orientable flat $n$-manifold $F$ for which every holonomy form $f$ satisfies
    \[
    \text{$(d(f),-1)_q \varepsilon_q(f)=\begin{cases}
        (3,3)_q& n=35\\
        1 & n\geq 39
    \end{cases}$ for all primes $q$.}
    \]
    In particular, $F$ has the UCC property and appears as a cusp cross-section in the commensurability class of arithmetic hyperbolic manifolds defined by
    \[q(x_1, \dots, x_{n+2})=\begin{cases}
   3x_1^2+ 3x_2^2+x_3^2 +\dots + x_{36}^2-x_{37}^2 & n=35 \\
    x_1^2+ x_2^2 +\dots + x_{n+1}^2-x_{n+2}^2 & n\geq 39
    \end{cases}.\]
\end{lemma}    
\begin{proof}
    Let $E$ be a flat manifold of dimension $12+4\ell$ as constructed in Lemma~\ref{lem:Emanifold}. Let $\beta : \wt{E} \rightarrow \wt{E}$ be the non-trivial deck transformation for a connected double cover $\wt{E}\rightarrow E$ with $b_1(\wt{E})=0$. Let $C$ be a $10+2k$-dimensional flat manifold as constructed in Lemma~\ref{lem:C3} with holonomy $(\Z/3\Z)^3$. This admits a regular 3-fold cover $\wt{C}\rightarrow C$ with $b_1(\wt{C})=0$. Let $\alpha: \wt{C} \rightarrow \wt{C}$ be a non-trivial deck transformation for this covering. Let $B_{HW}$ be the 3-dimensional Hantzsche-Wendt manifold with $\tau:B_{HW}\rightarrow B_{HW}$ the order three isometry constructed above. We construct $F$ by taking the quotient of $\wt{C}^{\times 2}\times \wt{E}\times B_{HW}$ by the action of $\Z/6\Z$ where the generator acts by
    \begin{align}\begin{split}
    \wt{C}\times \wt{C}\times \wt{E}\times B_{HW}&\rightarrow \wt{C}\times \wt{C}\times \wt{E}\times B_{HW}\\
    g(x,y,z,w)&= (\alpha(y),\alpha(x),\beta(z),\tau(w)).
    \end{split}
    \end{align}
    It can be verified that non-trivial elements of $\Z/6\Z$ act without fixed points. The isometries $g$, $g^3$ and $g^5$ have no fixed points since $\beta$ has no fixed points and the isometries $g^2$ and $g^4$ have no fixed points since $\alpha$ has no fixed points. Thus if we take $F$ to be the quotient of $\wt{C}^{\times 2}\times \wt{E}\times B_{HW}$ by the action of $\Z/6\Z$, then $F$ is a flat manifold of dimension $35+4(k+\ell)$. 
    
    Since $E$ and $C$ are orientable, the maps $\alpha$ and $\beta$ are orientation preserving. The map $\tau$ is clearly orientation preserving. Combined with the fact that $C$ is even dimensional, this implies that the generator of $\Z/6\Z$ acts by an orientation preserving isometry. Thus $F$ is orientable.

    Next, we calculate the holonomy forms of $F$. Since $b_1(\wt{C})=b_1(\wt{E})=b_1(B_{HW})=0$, Lemma~\ref{lem:product_hol_refined} and the covering
    \[
    \wt{C}\times \wt{C}\times \wt{E}\times B_{HW}\rightarrow F
    \]
    imply that there is a holonomy representation for $F$ taking values in $O(f;\Q)$ only if $f$ is of the form
\[
f=g_1 \oplus g_2 \oplus h \oplus \langle a_1,a_2,a_3 \rangle,
\]
where $g_1,g_2$ are holonomy forms for $\wt{C}$, $h$ is a holonomy form for $\wt{E}$ and $\langle a_1,a_2,a_3 \rangle$ is a holonomy form for $B_{HW}$. Now we apply the fact that $f$ is invariant under the action of $\Z/6\Z$. First we observe that $h$ is invariant under the action of $\beta$, and is thus a holonomy form for $E$. Since the $\Z/6\Z$ exchanges the $\wt{C}$ factors, we see that $g_1=g_2$. Since the isometry $\tau$ is induced by cyclically permuting the coordinates of $\R^3/\Z^3$, we see that $a_1=a_2=a_3$. Thus any holonomy form $f$ for $F$ takes the form 
\[
f=g \oplus g \oplus h \oplus \langle a,a,a \rangle,
\]
where $g$ is a holonomy form for $\wt{C}$, $h$ is a holonomy form for $E$ and $a\in \Q_{>0}$. This allows us to calculate
\begin{align*}
        (d(f),-1)_q \varepsilon_q(f)&=(a,-1)_q\varepsilon_q(g \oplus g \oplus h \oplus \langle a,a,a \rangle)\\
        &=(a,-1)_q\varepsilon_q(g \oplus g  \oplus \langle a,a,a \rangle)\\
        &=(a,-1)_q\varepsilon_q(g \oplus g)\varepsilon_q( \langle a,a,a \rangle)\\
        &=(a,-1)_q(d(g),d(g))_q(a,a)_q\\
        &=(d(g),d(g))_q.
    \end{align*}
    
However Proposition~\ref{prop:modpholonomy} shows that $d(g)=\begin{cases}
    3&k\equiv 0\bmod 2\\
    1 &k\equiv 1\bmod 2
\end{cases}$. Therefore, taking $k=\ell=0$ yields $F$ with the desired holonomy forms in dimension $n=35$. Taking $k=1$ and an arbitrary $\ell\geq 0$  yields $F$ with the desired holonomy forms for $n\geq 39$.
By Proposition~\ref{prop:projective_invariants} this means that the holonomy forms for $F$ all lie in a single projective equivalence class. By Proposition~\ref{prop:UCCcharacterization}, this implies that $F$ has the UCC property. One can calculate that
     \[f(x_1, \dots, x_{n})=\begin{cases}
   3x_1^2+ 3x_2^2+x_3^2 +\dots + x_{35}^2 & n=35 \\
    x_1^2+ x_2^2 +\dots + x_{n}^2 & n\geq 39
    \end{cases}.\]
     is a representative of the projective equivalence class of holonomy forms for $F$. By Theorem~\ref{thm:realization}, it follows that $F$ appears as a cusp cross-section in the required commensurability class of arithmetic hyperbolic manifolds.
\end{proof}

\section{Proof of Theorem~\ref{thm:manyUCCclasses} and Theorem~\ref{thm:3mod4examples}}

Using the manifolds $E$ constructed in Lemma~\ref{lem:Emanifold} and $B_p$ constructed in Lemma~\ref{lem:Bp} we construct the manifolds with the UCC property necessary to prove Theorem~\ref{thm:3mod4examples}. The key case is that of $n\equiv 0 \bmod 4$.
\begin{lemma}\label{lem:extensions_of_E}
Let $p\equiv 3\bmod 4$ be a prime. Then for any $n\equiv 0\bmod 4$ and $n\geq 2p^2+2p+8$, there is an orientable flat $n$-manifold $E_p$ such that every holonomy form $f$
 for $E_p$ satisfies
 \[\text{$d(f)=1$ and $\varepsilon_q(f)=(p,p)_q=\begin{cases}
     -1& q=2,p\\
     1& q\ne 2,p
 \end{cases}$ for any prime $q$.}\]
 \end{lemma}
\begin{proof}
    Take $E$ to be a flat manifold of dimension $12+4k$ as constructed in Lemma~\ref{lem:Emanifold}. Let $\alpha: \wt{E}\rightarrow \wt{E}$ be the non-trivial deck transformation for the double cover $\widetilde{E}\rightarrow E$ with $b_1(\wt{E})=0$. Let $B_p$ be a manifold of dimension $p^2+p-2$ with holonomy group $(\Z/p\Z)^2$ and $b_1(B_p)=0$ as constructed in Lemma~\ref{lem:Bp}. There is an action of $\Z/2\Z$ on $\wt{E}\times B_p\times B_p$, where the non-trivial element of $\Z/2\Z$ acts by
    \begin{align*}
        \wt{E}\times B_p\times B_p &\rightarrow \wt{E}\times B_p\times B_p \\
        (x,y,z)&\mapsto (\alpha(x), z,y).
    \end{align*}
    Since $\alpha$ acts without fixed points on $\wt{E}$ we have that $\Z/2\Z$ acts without fixed points. Thus if we take $E_p$ to be the quotient of $\wt{E}\times B_p\times B_p$ by this action of $\Z/2\Z$ then $E_p$ is a compact flat manifold of dimension $2p^2+2p+8+4k$. Since $\alpha$ preserves orientations and $B_p$ is orientable and of even dimension, we see that $E_p$ is orientable.
    
    Let $f$ be a holonomy form for $E_p$. Given the covering
    \[
    \wt{E}\times B_p\times B_p\rightarrow E_p
    \]
    we see that $f$ takes the form
    \[
    f=g\oplus h_1 \oplus h_2,
    \]
    where $g$ is a holonomy form for $\wt{E}$ and $h_1,h_2$ are holonomy forms for $B_p$. Since $f$ is invariant under the action of $\Z/2\Z$, we see that $g$ is invariant under the action of $\alpha$, implying that it is actually a holonomy form for $E$, and that $h_1=h_2$. Since $\dim B_p\equiv 2\bmod 4$, Proposition~\ref{prop:modpholonomy} implies that $d(h_1)=p$. Thus we may calculate that $f$ satisfies $d(f)=1$ and that for all primes $q$ we have
    \[
    \varepsilon_q(f)=\varepsilon_q(h_1\oplus h_1)=(p,p)_q,
    \]
    as required.
\end{proof}

\begin{proof}[Proof of Theorem~\ref{thm:3mod4examples}]
Take $E_p$ to be a manifold of dimension $n\equiv 0\bmod 4$ and $n\geq 2p^2+2p+8$ as constructed in Lemma~\ref{lem:extensions_of_E}. Let $F$ be a 39-dimensional manifold as constructed in Lemma~\ref{lem:3mod4UCC}. Using Lemma~\ref{lem:product_hol1} and Proposition~\ref{prop:projective_invariants}, one can calculate that the flat manifolds of the form $E_p$, $E_p\times S^1$ and $E_p\times F$ all have a unique projective equivalence class of holonomy forms. Consequently these manifolds all have the UCC property by Proposition~\ref{prop:UCCcharacterization}. Moreover, one can check that for each of these families, the unique projective equivalence class of holonomy forms contains:
\[
f(x_1,\dots, x_n)=px_1^2 + px_2^2+x_3^2+\dots + x_n^2. 
\]
These give examples in dimensions $n\not\equiv 2 \bmod 4$. By Theorem~\ref{thm:realization}, this implies that the manifolds $E_p$, $E_p\times S^1$ and $E_p\times F$ all appear as cusp cross-sections in the commensurability class of arithmetic hyperbolic manifolds defined by the form $q$ in the statement of the theorem.

Let $E$ be a manifold of dimension $n\equiv 0\bmod 4$ as constructed in Lemma~\ref{lem:Emanifold}. Let $B_p$ be a manifold of dimension $p^2+p-2$ with holonomy group $(\Z/p\Z)^2$ as constructed in Lemma~\ref{lem:Bp}. Using Theorem~\ref{thm:p3mod4UCCexamples}, we can calculate that manifolds of the form $E\times B_p$ have a unique projective equivalence class of holonomy form and that this class is represented by
\[
f(x_1,\dots, x_n)=\begin{cases}
    px_1^2 + x_2^2+\dots + x_n^2 &n\equiv 2 \bmod 8\\
    px_1^2 + px_2^2+px_3^2+x_4^2+\dots + x_n^2 &n \equiv 6\bmod 8
\end{cases}
\]
By Theorem~\ref{thm:realization} and Proposition~\ref{prop:UCCcharacterization}, this means that manifolds of the form $B_p\times E$ have the UCC property and appear as cusp cross-sections in the required commensurability class.
\end{proof}
\begin{proof}[Proof of Theorem~\ref{thm:manyUCCclasses}]
    This follows from Theorem~\ref{thm:3mod4examples}. It suffices to observe that there are infinitely many primes $p\equiv 3 \bmod 4$ and that for distinct primes the quadratic forms appearing in Theorem~\ref{thm:3mod4examples} are not projectively equivalent. To see this, let
    \[q(x_1, \dots, x_{n+2})=\begin{cases}
   px_1^2+ px_2^2+x_3^2 +\dots + x_{n+1}^2-x_{n+2}^2 & n\not\equiv 2\bmod 4 \\
    px_1^2+ x_2^2 +\dots + x_{n+1}^2-x_{n+2}^2 & n\equiv 2\bmod 8\\
    px_1^2+ px_2^2+px_3^2+x_4^2 +\dots + x_{n+1}^2-x_{n+2}^2 & n\equiv 6\bmod 8.
    \end{cases}\]
    for some $p\equiv 3\bmod 4$.
    When $n\equiv 2 \bmod 4$, one has that $d(q)=-p$. Since the discriminant is projective equivalence invariant for even-dimensional forms, this handles the case $n\equiv 2 \bmod 4$.
    When $n\not\equiv 2\bmod 4$, we have that $d(q)=-1$ and for any prime $\ell$,
    \[
        \varepsilon_\ell (q)=(p,p)_\ell=\begin{cases}
            -1 &\ell=2,p\\
            1 &\text{otherwise}.
        \end{cases}
    \]
    Since $q$ has rank $n+2$, it follows from Proposition~\ref{prop:projective_invariants} that $\varepsilon_\ell(q)$ is a projective invariant of $q$ for all primes $\ell$ whenever $n\equiv 0,3 \bmod{4}$. If $n\equiv 1\bmod 4$, then the relevant projective invariant is
    \[
    (d(q),-1)_\ell\varepsilon_\ell (q)=(-1,-1)_\ell(p,p)_\ell=\begin{cases}
            -1 &\ell=p\\
            1 &\text{otherwise}.
        \end{cases}
    \]
    Thus for any $n$, we see that distinct primes give rise to distinct projective equivalence classes of quadratic forms as required.
\end{proof}

\section{Non-arithmetic cusp pairs}
We conclude by constructing non-arithmetic cusp pairs in small dimensions. We begin with even dimensions, where the pairs can be distinguished using the discriminant of the holonomy form.
\begin{lemma}\label{lem:non-arithpairs-even_dims}
For every dimension satisfying $n\geq 10$ and $n\equiv 2 \bmod 4$ there is an orientable non-arithmetic cusp pair of dimension $n$. 
For every dimension satisfying $n\geq 16$ and $n\equiv 0 \bmod 4$ there is a non-arithmetic cusp pair of dimension $n$. Moreover, for $n\geq 20$, this can be taken as an orientable non-arithmetic cusp pair.
\end{lemma}
\begin{proof}
Let $C^k$ denote a $k$-dimensional flat manifold of type $C$ as built in Lemma~\ref{lem:C2} where $k\geq 6$ is even. Note that $C^k$ is orientable for $k\geq 10$. Let $B_3^\ell$ denote an $\ell$-dimensional flat manifold of type $B_p$ with $p=3$ as built in Lemma~\ref{lem:Bp} where $\ell \geq 10$ and $\ell\equiv 2\bmod 4$. With these choices any holonomy form for $B_3^\ell$ satisfies $d(f)=3$ and any holonomy form for $C^k$ satisfies $d(f)=1$. This implies that $B_3^n$ and $C^n$ do not have any holonomy forms in common and so form a non-arithmetic cusp pair by Proposition~\ref{prop:non_arith_pairs_characterization}. This gives an orientable non-arithmetic cusp pair in dimensions $n\geq 10$ with $n\equiv 2\bmod 4$.

For $n\equiv 0\bmod 4$ and $n\geq 16$, we see that $C^{n-10}\times B_3^{10}$ and $C^{n}$ form a non-arithmetic cusp pair. Since every holonomy form for $C^{n-10}\times B_3^{10}$ satisfies $d(f)=3$ and every holonomy form for $C^{n}$ satisfies $d(f)=1$, we see that $C^{n-10}\times B_3^{10}$ and $C^{n}$ do not have any holonomy forms in common. By Proposition~\ref{prop:non_arith_pairs_characterization}, this implies that $C^{n-10}\times B_3^{10}$ and $C^{n}$ form a non-arithmetic cusp pair. This gives an orientable non-arithmetic cusp pair when $n\geq 20$ since $C^{k}$ is orientable for $k\geq 10$. 
\end{proof}

Next we consider some mapping tori. These will be used to construct odd-dimensional non-arithmetic cusp pairs whose holonomy forms are distinguished by the Hasse-Witt invariants at the prime $p=2$. 
\begin{lemma}\label{lem:7and15mapping_tori}
For integers $k,\ell\geq 0$, there exists a flat manifold $M_{k,\ell}$ of dimension $n=6k+8\ell +1 $ such that every holonomy form $f$ for $M_{k,\ell}$ satisfies
    \[
    \left(d(f), (-1)^{\frac{n-1}{2}}\right)_2 \varepsilon_2(f)= \begin{cases}
        1 & n\equiv 1,7 \bmod 8\\
        -1 &n\equiv 3,5 \bmod 8.
    \end{cases}
    \]   
\end{lemma}
\begin{proof}
Let $\alpha: T^6 \rightarrow T^6$ be an isometry of order seven and let $\beta : T^{8} \rightarrow T^{8}$ be an isometry of order 15. Let $M_{k,\ell}$ be the $6k+8\ell +1$ flat manifold obtained as the mapping torus of the map
\[
\alpha^{\times k} \times \beta^{\times \ell} : T^{6k+8\ell} \rightarrow T^{6k+8\ell}.
\]
The holonomy representation of $M_{k,\ell}$ is of the form
    \[
    \rho\cong \sigma_7^{\oplus k} \oplus \sigma_{15}^{\oplus \ell} \oplus \mathbf{1},
    \]
    where $\mathbf{1}$ denotes a trivial representation.
    Note that $\sqrt{-7}\in \Q[\zeta_7]$ and that $\sqrt{-15}\in \Q[\zeta_{15}]$. For this latter fact it suffices to observe that $\sqrt{-3}\in \Q[\zeta_{3}]$ and $\sqrt{5}\in \Q[\zeta_{5}]$ and observe that $\Q[\zeta_{3}]$ and $\Q[\zeta_{5}]$ are both contained in $\Q[\zeta_{15}]$. Since $-7\equiv -15\equiv 1 \bmod 8$, we have that $-7$ and $-15$ are squares in $\Q_2$. Thus Proposition~\ref{prop:cyclic_rep_calc} shows that over $\Q_2$ any holonomy form for $M_{k,\ell}$ is equivalent to
    \[
    f\sim \langle a, \underbrace{1,\dots, 1}_{(n-1)/2}, \underbrace{-1,\dots, -1}_{(n-1)/2} \rangle.
    \]
    From this, we calculate that
    \begin{align*}
        \left(d(f), (-1)^{\frac{n-1}{2}}\right)_2 \varepsilon_2(f)&= \left((-1)^{\frac{n-1}{2}}a, (-1)^{\frac{n-1}{2}}\right)_2 \left(a, (-1)^{\frac{n-1}{2}}\right)_2  (-1)^{\frac{(n-1)(n-3)}{8}}\\
        &=(-1)^{\frac{n^2-1}{8}} 
    \end{align*}    
\end{proof}
The following lemma taken with Lemma~\ref{lem:non-arithpairs-even_dims} easily gives Theorem~\ref{thm:non-arith_pairs}.
\begin{lemma}\label{lem:odd_dim_non-arith_pairs}
    If $n$ is an integer such that either
    \begin{enumerate}
        \item $n=13$ or $19$,
        \item $n\equiv 1 \bmod 4$ and $n\geq 21$ or
        \item $n\equiv 3 \bmod 4$ and $n\geq 27$,
    \end{enumerate}
    then there exists an orientable non-arithmetic cusp pair in dimension $n$.
\end{lemma}
\begin{proof}
    For $m\equiv 0 \bmod 4$ and $m\geq 12$, we take $E^m$ to be an $m$-dimensional manifold as constructed in Lemma~\ref{lem:Emanifold}. For $k,\ell\geq 0$ we take $M_{k,\ell}$ to be the $6k+8\ell +1$-dimensional flat manifold as constructed in Lemma~\ref{lem:7and15mapping_tori}.

    First suppose that $n\equiv 1 \bmod 4$. For $n\geq 13$, the manifold $S^1\times E^{n-1}$ is a flat manifold for which every holonomy form satisfies $\varepsilon_2(f)=1$. Thus we have a non-arithmetic cusp pair whenever we can produce a flat $n$-manifold $N$ for which every holonomy form satisfies $\varepsilon_2(f)=-1$. For $n=13$, we take $N=M_{2,0}$. For $n=21$, we take $N=M_{2,1}$. For $n\geq 25$, we may take $N=M_{2,0}\times E^{n-13}$.

    Now suppose that $n\equiv 3 \bmod 4$ and $n\geq 19$. The manifold $M_{1,0}\times E^{n-7}$ is a flat manifold such that every holonomy form satisfies $(d(f),-1)_2\varepsilon_2(f)=1$. Thus we have a non-arithmetic cusp pair whenever we can find an $n$-manifold $N$ for which every holonomy form satisfies $(d(f),-1)_2\varepsilon_2(f)=-1$. For $n=19$, we take $N=M_{3,0}$. For $n=27$, we take $N=M_{3,1}$. For $n\geq 31$, we take $N=M_{3,0}\times E^{n-19}$.
\end{proof}

A slightly more elaborate construction yields a 16-dimensional orientable non-arithmetic cusp pair.
\begin{lemma}\label{lem:16dim_non_arith_pair}
    There exists an orientable non-arithmetic cusp pair in dimension 16. 
\end{lemma}
\begin{proof}
  Let $E$ be a $16$-dimensional manifold as constructed in Lemma~\ref{lem:Emanifold}. Such an $E$ has the UCC property and every holonomy form for $E$ satisfies $d(f)=1$ and $\varepsilon_p(f)=1$ for all primes $p$.

  Let $\alpha: T^6 \rightarrow T^6$ be an isometry of order seven on the 6-torus. Let $B_{HW}$ be the 3-dimensional Hantzsche-Wendt manifold equipped with $\tau : B_{HW}\rightarrow B_{HW}$ the orientation-preserving isometry of order three obtained by cyclically permuting coordinates, as constructed prior to Lemma~\ref{lem:3mod4UCC}. Let $M$ be the orientable flat manifold obtained by taking the mapping torus of
  \[
  \alpha \times \alpha \times \tau : T^6\times T^6\times B_{HW} \rightarrow T^6\times T^6\times B_{HW}.
  \]
  By construction, the holonomy representation of $M$ takes the form
  \[
  \rho= \rho_1 \oplus \rho_1 \oplus \rho_2 \oplus \mathbf{1}, 
  \]
  where $\rho_1$ factors through $\sigma_7$, the non-trivial rational representation of $\Z/7\Z$, and $\rho_2$ factors through the non-trivial 3-dimensional representation of $A_4$ with image generated by matrices
  \begin{equation}\label{eq:16dim}
      \text{$\begin{pmatrix}
          -1 & 0& 0\\
          0& -1 & 0\\
          0&0&1
      \end{pmatrix}$ and $\begin{pmatrix}
          0 & 1& 0\\
          0& 0 & 1\\
          1&0&0
      \end{pmatrix}$}
  \end{equation}
   
   In fact, one can identify the holonomy group as being isomorphic to $\Z/7\Z \times A_4$. Using Proposition~\ref{prop:form_decomposition}, any holonomy form for $M$ is equivalent to
  \[
  g\sim g_1 \oplus g_2 \oplus g_3 \oplus \langle a \rangle, 
  \]
  where $O(g_1;\Q)$ and $O(g_2;\Q)$ contain the image of $\sigma_7$ and $O(g_3;\Q)$ contains the image of a non-trivial 3-dimensional representation of $A_4$. Since $-7$ is a square in $\Q_2$, we see that $\varepsilon_2(g_1)=\varepsilon_2(g_2)=-1$ and that $d(g_1)=d(g_2)=7$. Using the fact $O(g_3;\Q)$ contains the matrices in \eqref{eq:16dim}one can calculate that $g_3$ must take the form $g_3=\langle b,b,b\rangle$ for some $b\in\Q_{>0}$.
  Using Proposition~\ref{prop:sum_of_forms}, this allows us to calculate $\varepsilon_2(g)$ as
  \begin{align*}
      \varepsilon_2(g)&=\varepsilon_2(g_1 \oplus g_2)\varepsilon_2(\langle a, b,b,b\rangle)\\
      &=(7,7)_2 (a,b)_2(b,b)_2\\
      &=-(d(g),b)_2.
  \end{align*}
  In particular, if $d(g)=1$, then $\varepsilon_2(g)=-1$. Since every holonomy form $f$ for $E$ satisfies $d(f)=1$ and $\varepsilon_2(f)=1$, it follows that $E$ and $M$ have no holonomy forms in common. It follows from Proposition~\ref{prop:non_arith_pairs_characterization} that $M$ and $E$ form a non-arithmetic cusp pair.
\end{proof}

\bibliographystyle{alpha}
\bibliography{bib}

@article{FarrellZ,
    author = {Farrell, F. Thomas and Zdravkovska, Smilka},
    title = {Do almost flat manifolds bound?},
    journal = {Michigan Math. Journ.},
    fjournal = {Michigan Mathematical Journal},
    volume = {30},
    year = {1987},
    number = {2},
    pages = {199--208}
}

@misc{McCoySell2024,
	author = {McCoy, D. and Sell, C.},
	howpublished = {arXiv:2410.10707},
	title = {Cusp types of arithmetic hyperbolic manifolds},
	year = {2024}}

@misc{GrimbertMcCoySell,
	author = {Grimbert, M. and McCoy, D. and Sell, C.},
	howpublished = {arXiv:2609.08756},
	title = {Cusp cross-sections of low-dimensional arithmetic hyperbolic manifolds},
	year = {2026}}

@article {Gromov,
    AUTHOR = {Gromov, M. and Piatetski-Shapiro, I.},
     TITLE = {Nonarithmetic groups in {L}obachevsky spaces},
   JOURNAL = {Inst. Hautes \'Etudes Sci. Publ. Math.},
  FJOURNAL = {Institut des Hautes \'Etudes Scientifiques. Publications
              Math\'ematiques},
    NUMBER = {66},
      YEAR = {1988},
     PAGES = {93--103},
      ISSN = {0073-8301,1618-1913},
   MRCLASS = {22E40},
  MRNUMBER = {932135},
MRREVIEWER = {Gopal\ Prasad},
       URL =
              {http://www.numdam.org/numdam-bin/item?id=PMIHES_1987__66__93_0},
}

@article {Long-Reid2000geometric,
    AUTHOR = {Long, D. D. and Reid, A. W.},
     TITLE = {On the geometric boundaries of hyperbolic {$4$}-manifolds},
   JOURNAL = {Geom. Topol.},
  FJOURNAL = {Geometry and Topology},
    VOLUME = {4},
      YEAR = {2000},
     PAGES = {171--178},
      ISSN = {1465-3060,1364-0380},
   MRCLASS = {57R90 (57M50 58J28)},
  MRNUMBER = {1769269},
MRREVIEWER = {Marko\ Kranjc},
       DOI = {10.2140/gt.2000.4.171},
       URL = {https://doi.org/10.2140/gt.2000.4.171},
}

@article {Vasquez1970flat,
    AUTHOR = {Vasquez, A. T.},
     TITLE = {Flat {R}iemannian manifolds},
   JOURNAL = {J. Differential Geometry},
  FJOURNAL = {Journal of Differential Geometry},
    VOLUME = {4},
      YEAR = {1970},
     PAGES = {367--382},
      ISSN = {0022-040X,1945-743X},
   MRCLASS = {53.72},
  MRNUMBER = {267487},
MRREVIEWER = {S.\ Takizawa},
       URL = {http://projecteuclid.org/euclid.jdg/1214429510},
}

@article {Hiller1986flat,
    AUTHOR = {Hiller, H. and Sah, C.-H.},
     TITLE = {Holonomy of flat manifolds with {$b_1=0$}},
   JOURNAL = {Quart. J. Math. Oxford Ser. (2)},
  FJOURNAL = {The Quarterly Journal of Mathematics. Oxford. Second Series},
    VOLUME = {37},
      YEAR = {1986},
    NUMBER = {146},
     PAGES = {177--187},
      ISSN = {0033-5606,1464-3847},
   MRCLASS = {53C20 (20F38 57S30)},
  MRNUMBER = {841425},
MRREVIEWER = {F.\ E. A. Johnson},
       DOI = {10.1093/qmath/37.2.177},
       URL = {https://doi.org/10.1093/qmath/37.2.177},
}

@article {LongReid,
    AUTHOR = {Long, D. D. and Reid, A. W.},
     TITLE = {All flat manifolds are cusps of hyperbolic orbifolds},
   JOURNAL = {Algebr. Geom. Topol.},
  FJOURNAL = {Algebraic \& Geometric Topology},
    VOLUME = {2},
      YEAR = {2002},
     PAGES = {285--296},
      ISSN = {1472-2747,1472-2739},
   MRCLASS = {57M50 (11E20 11F06 20E26 30F40)},
  MRNUMBER = {1917053},
MRREVIEWER = {John\ G.\ Ratcliffe},
       DOI = {10.2140/agt.2002.2.285},
       URL = {https://doi.org/10.2140/agt.2002.2.285},
}

@book {Serre_arithmetic,
    AUTHOR = {Serre, J.-P.},
     TITLE = {A course in arithmetic},
    SERIES = {Graduate Texts in Mathematics},
    VOLUME = {No. 7},
      NOTE = {Translated from the French},
 PUBLISHER = {Springer-Verlag, New York-Heidelberg},
      YEAR = {1973},
     PAGES = {viii+115},
   MRCLASS = {12-02 (10CXX 10DXX)},
  MRNUMBER = {344216},
}

@article {McReynolds2009covers,
    AUTHOR = {McReynolds, D. B.},
     TITLE = {Controlling manifold covers of orbifolds},
   JOURNAL = {Math. Res. Lett.},
  FJOURNAL = {Mathematical Research Letters},
    VOLUME = {16},
      YEAR = {2009},
    NUMBER = {4},
     PAGES = {651--662},
      ISSN = {1073-2780},
   MRCLASS = {57N16 (22F50 57M50)},
  MRNUMBER = {2525031},
MRREVIEWER = {Darren\ D.\ Long},
       DOI = {10.4310/MRL.2009.v16.n4.a8},
       URL = {https://doi.org/10.4310/MRL.2009.v16.n4.a8},
}

@article {Sell,
    AUTHOR = {Sell, C.},
     TITLE = {Cusps and commensurability classes of hyperbolic 4-manifolds},
   JOURNAL = {Algebr. Geom. Topol.},
  FJOURNAL = {Algebraic \& Geometric Topology},
    VOLUME = {23},
      YEAR = {2023},
    NUMBER = {8},
     PAGES = {3805--3834},
      ISSN = {1472-2747},
   MRCLASS = {57K40 (11E20 11F06 16H05 57K50 57M50)},
  MRNUMBER = {4665284},
       DOI = {10.2140/agt.2023.23.3805},
       URL = {https://doi.org/10.2140/agt.2023.23.3805},
}

\end{document}